\documentclass[14pt]{article}

\usepackage{mathtext}
\usepackage[T1,T2A]{fontenc}
\usepackage[utf8]{inputenc}
\usepackage[english]{babel}

\usepackage{amsfonts, amsmath, amssymb, amsthm, amscd}
\usepackage{mathtools}
\usepackage{xcolor}
\usepackage{ulem}

\usepackage{geometry}
\usepackage{graphicx}
\usepackage{hyperref}
\usepackage{cancel}
\hypersetup{pdfstartview=FitH,  linkcolor=blue,urlcolor=urlcolor, citecolor=blue, colorlinks=true}

\theoremstyle{theorem}
\newtheorem{theorem}{Theorem}[section]
\theoremstyle{theorem}
\newtheorem{proposition}{Proposition}[section]
\theoremstyle{definition}
\newtheorem{definition}{Definition}[section]
\theoremstyle{lemma}
\newtheorem{lemma}{Lemma}[section]

\newcommand{\Label}[1]{\label{#1}}

\DeclarePairedDelimiter\abs{\lvert}{\rvert}

\usepackage{fancyhdr}
\title{Pseudo-Differential Equations with Weak Degeneration for Radial Functions of $p$-adic Argument}

\author{\textbf{Alexandra V. Antoniouk}\\
\footnotesize Institute of Mathematics,\\
\footnotesize National Academy of Sciences of Ukraine,\\
\footnotesize Tereshchenkivska 3, Kyiv, 01024 Ukraine,\\
\footnotesize E-mail: antoniouk.a@gmail.com
\and
\textbf{Anatoly N. Kochubei}\\
\footnotesize Institute of Mathematics,\\
\footnotesize National Academy of Sciences of Ukraine,\\
\footnotesize Tereshchenkivska 3, Kyiv, 01024 Ukraine,\\
\footnotesize E-mail: kochubei@imath.kiev.ua
\and
\textbf{Mariia V. Serdiuk}\\
\footnotesize Taras Shevchenko National University of Kyiv,\\
\footnotesize  Akademika Hlushkova Ave 4-е, Kyiv, 03127 Ukraine,\\
\footnotesize E-mail: mariia.v.serdiuk@gmail.com }

\begin{document}
\maketitle

\begin{abstract}
In paper A. N. Kochubei (Pacif. J. Math. 269 (2014), 355--369) a right inverse to Vladimirov-Taibleson fractional differential operator $D^{\alpha},\;\alpha>0$ was found. It turns out that this permits to reduce the $p$-adic Cauchy problem for radial functions to an integral equation whose properties resemble those of classical Volterra equations.

In this article, we develop further these ideas to an important class of degenerate pseudo-differential equations with operator $D^{\alpha},\;\alpha>0$.
We find the conditions of local and global solvability for corresponding nonlinear weakly degenerate equations.

\end{abstract}

\section{Introduction}

It is known that the operation of differentiation is not defined for mappings $\mathbb{Q}_p\to\mathbb{C}$, and the Vladimirov-Taibleson pseudo-differential operator $D^{\alpha}$ plays a role of a differential operator in the $p$-adic analysis. Pseudo-differential equations of such kind are widely used and have many applications in mathematical physics, e.g. \cite{Albeverio}, \cite{Vlad}, \cite{Kozyrev}. 
However, only in the works by Kochubei \cite{K1,K2} it was noted that the properties of the Vladimirov operator on radial functions are significantly simplified.
In \cite{K2} the nonlinear Cauchy problem 
\begin{equation}\Label{00-1}\
(D^{\alpha}u)(|t|_p)=f(|t|_p,u(|t|_p)),\;0\neq t\in \mathbb{Q}_p,
\end{equation}
with the initial condition
\begin{equation}\Label{00-2}\
u(0)=u_0
\end{equation}
was studied.
Here
\begin{equation}\Label{1-7}
(D^{\alpha}\varphi)(t)=\frac{1-p^{\alpha}}{1-p^{-\alpha-1}}\int\limits_{\mathbb{Q}_p}|y|_p^{-\alpha-1}\left[\varphi(t-y)-\varphi(y)\right]dy.
\end{equation}
 is the Vladimirov operator of fractional differentiation on a non-Archimedean local field $\mathbb{Q}_p$.

Generally speaking the operator $D^{\alpha}$ on $L^2(\mathbb{Q}_p)$ has a rather complicated structure. In particular, it has a point spectrum of infinite multiplicity, and it is well known that the corresponding ``evolution equation'' with this operator in the p-adic time variable $t$ does not have a fundamental solution.

At the same time if one considers only solutions depending on $\vert t\vert_p$ (see \cite{K1}), the equation (\ref{00-1}) may be considered as a $p$-adic analog of ordinary differential equation. In this case, the right inverse $I^\alpha$ to $D^\alpha$ gives a $p$-adic counterpart of the Riemann-Liouville fractional integral or, for $\alpha =1$, the classical antiderivative. Therefore the Cauchy problem for a (linear or nonlinear) equation with $D^{\alpha}$ on the radial restriction is reduced to an integral equation with properties close to those of classical Volterra equations. We can say that in \cite{K1,K2} a $p$-adic form of ordinary differential equations was found in the framework of $p$-adic pseudo-differential calculus.

Under appropriate conditions \cite{K2}, the nonlinear integral equation corresponding to (\ref{00-1})-(\ref{00-2}) is a) locally solvable, b) its solution is extended to all $t\in \mathbb Q_p$, and c) the solutions satisfy the Cauchy problem.

In this paper, we study a class of degenerate equations
\begin{equation}\Label{00-3}\
|t|_p^{\gamma}(D^{\alpha}u)(|t|_p)=f(|t|_p,u(|t|_p)),\;0\neq t\in \mathbb{Q}_p,
\end{equation}
with $\gamma>0$, and show that the methods from \cite{K1,K2} can be extended, with appropriate modifications, to the case of weak degeneration for $\gamma < \min (1,\alpha )$. 

\section{Preliminaries}

Let $\mathbb{Q}_p$ be the field of $p$-adic numbers which is defined as the completion of the field of rational numbers $\mathbb{Q}$ with respect to the non-Archimedean $p$-adic norm $|\cdot|_p$. For a rational number $x=p^{\gamma}(m/n)$, where the integers $m$ and $n$ are prime to $p$, the $p$-adic norm of $x$ equals $|x|_p=p^{-\gamma}$. This norm satisfies the strong triangle inequality $|x+y|_p\leq\max\left(|x|_p,|y|_p\right)$.

A function $\varphi:\mathbb{Q}_p\to\mathbb{C}$ is said to be {\it locally constant} if there exists such an integer $\ell$ that for any $x\in\mathbb{Q}_p$
$$\varphi(x+x')=\varphi(x),\mbox{ whenever }|x'|_p\leq p^{-\ell}.$$
The smallest of such numbers $\ell$ is called the exponent of local constancy of the function $\varphi$.

We denote by $\mathcal{D}(\mathbb{Q}_p)$ the linear topological space of all locally constant functions $\varphi:\mathbb{Q}_p\to\mathbb{C}$ with compact supports. The strong conjugate space $\mathcal{D}'(\mathbb{Q}_p)$ is called the space of Bruhat-Schwartz distributions.

Let $\chi_p(\xi x)=\exp(2\pi i\{\xi x\}_p)$, where $\xi\in\mathbb{Q}_p$, be an additive character of $\mathbb{Q}_p$. The Fourier transform of a test function $\varphi\in\mathcal{D}(\mathbb{Q}_p)$ is defined as
$$\widetilde{\varphi}(\xi)=\int\limits_{\mathbb{Q}_p}\chi_p(\xi x)\varphi(x)dx,\;x\in\mathbb{Q}_p.$$

The fractional differential operator $D^{\alpha},\;\alpha>0$ \eqref{1-7}, on $\mathcal{D}(\mathbb{Q}_p)$ is also defined as (see e.g. \cite{Vlad}):
\begin{equation}\Label{1-6}
(D^{\alpha}\varphi)(x)=\mathcal{F}^{-1}\big[|\xi|_p^{\alpha}(\mathcal{F}(\varphi))(\xi)\big](x).
\end{equation}

Further, we shall use the following integration formulas \cite{Vlad}:
\begin{align}\Label{0-11}
&\int\limits_{|x|_p\leq p^n}|x|_p^{\alpha-1}dx=\frac{1-p^{-1}}{1-p^{-\alpha}}p^{\alpha n};\mbox{ here and below }n\in\mathbb{Z},\;\alpha>0,\\
\Label{0-22}
&\int\limits_{|x|_p=p^n}|x-a|_p^{\alpha-1}dx=\frac{p-2+p^{-\alpha}}{p(1-p^{-\alpha})}|a|_p^{\alpha},\;|a|_p=p^n,\\
\Label{0-33}
&\int\limits_{|x|_p\leq p^n}\log|x|_pdx=\left(n-\frac{1}{p-1}\right)p^n \log p,\\
\Label{0-4}
&\int\limits_{|x|_p=p^n}\log|x-a|_pdx=\left[\left(1-\frac{1}{p}\right)\log|a|_p-\frac{\log p}{p-1}\right],\;|a|_p=p^n,\\
\Label{0-44}
&\int\limits_{|x|_p\leq p^n}dx=p^n;\;\int\limits_{|x|_p=p^n}dx=\left(1-\frac{1}{p}\right)p^n,
\end{align}
where $dx$ denotes the standard Haar measure on $\mathbb{Q}_p$, i.e. the invariant under shifts positive measure $dx$ such that $\int_{|x|_p\leq1}dx=1$.

Denote $d_{\alpha}=\frac{1-p^{\alpha}}{1-p^{-\alpha-1}}$. The following lemma states sufficient conditions when Vladimirov operator $D^\alpha$ on a radial function $u=u(|t|_p)$ is defined and gives an explicit expression for its action.

\begin{lemma}[\cite{K1}]\label{lemma1}
If a function $u=u(|x|_p)$ is such that
\begin{equation}\Label{3-1}
\sum\limits_{k=-\infty}^mp^k|u(p^k)|<\infty,
\end{equation}
\begin{equation}\Label{3-1'}
\sum\limits_{\ell=m}^{\infty}p^{-\alpha \ell}|u(p^{\ell})|<\infty,
\end{equation}
for some $m\in\mathbb{Z}$, then for any $n\in\mathbb{Z}$ the operator $(D^{\alpha}u)(|x|_p)$ defined in (\ref{1-7}) exists for $|x|_p=p^n$, depends only on $|x|_p$, and
\begin{align}\nonumber
(D^{\alpha}u)(p^n)=d_{\alpha}\left(1-\frac{1}{p}\right)p^{-(\alpha+1)n}\sum\limits_{k=-\infty}^{n-1}p^ku(p^k)+p^{-\alpha n-1}\frac{p^{\alpha}+p-2}{1-p^{-\alpha-1}}u(p^n)+\\
\nonumber
+d_{\alpha}\left(1-\frac{1}{p}\right)\sum\limits_{\ell=n+1}^{\infty}p^{-\alpha \ell}u(p^{\ell}).
\end{align}
\end{lemma}

In \cite{K1}, an analogue of the integral with a variable upper limit was introduced, which allowed to consider the problem (\ref{00-1}) - (\ref{00-2}) as $p$-adic counterpart of ODE.

\begin{definition}[\cite{K1}]
The $p$-{\it adic fractional integral} is defined for any function $\varphi\in\mathcal{D}(\mathbb{Q}_p)$ as follows:
\begin{equation}\Label{1-8}
(I^{\alpha}\varphi)(x)=\frac{1-p^{-\alpha}}{1-p^{\alpha-1}}\int\limits_{|y|_p\leq|x|_p}\big(|x-y|_p^{\alpha-1}-|y|_p^{\alpha-1}\big)\varphi(y)dy,\;\alpha\neq1,
\end{equation}
\begin{equation}\Label{1-9}
(I^1\varphi)(x)=\frac{1-p}{p\log p}\int\limits_{|y|_p\leq|x|_p}\big(\log|x-y|_p-\log|y|_p\big)\varphi(y)dy.
\end{equation}
\end{definition}

Below we give a set of important lemmas, which describe the main properties of this new type integral $I^{\alpha}$ acting on any radial function, needed for further considerations.

\begin{lemma}[\cite{K1}]\label{lemma2}
Suppose that
\begin{equation}\Label{2-7}\
\sum\limits_{k=-\infty}^m\max(p^k,p^{\alpha k})|u(p^k)|<\infty,\mbox{ if }\alpha\neq1,
\end{equation}
\begin{equation}\Label{2-8}\
\sum\limits_{k=-\infty}^m|k|\;p^k|u(p^k)|<\infty,\mbox{ if }\alpha=1,
\end{equation}
for some $m\in\mathbb{Z}$. Then $I^{\alpha}u$ exists, it is a radial function, and for any $x\neq0$:
$$
(I^{\alpha}u)(|x|_p)=p^{-\alpha}|x|_p^{\alpha}u(|x|_p)+\frac{1-p^{-\alpha}}{1-p^{\alpha-1}}\int\limits_{|y|_p<|x|_p}(|x|_p^{\alpha-1}-|y|_p^{\alpha-1})u(|y|_p)dy,\;\alpha\neq1,
$$
$$
(I^1u)(|x|_p)=p^{-1}|x|_pu(|x|_p)+\frac{1-p}{p\log p}\int\limits_{|y|_p<|x|_p}\left(\log|x|_p-\log|y|_p\right)u(|y|_p)dy.
$$
\end{lemma}

The following Lemma shows that on a certain class of radial functions the operator $I^{\alpha}$ acts as a right inverse of the operator $D^{\alpha}$.

\begin{lemma}[\cite{K1}]\label{lemma3}
Suppose that for some $m\in\mathbb{Z}$,
\begin{equation}\Label{3-2}\
\sum\limits_{k=-\infty}^m\max(p^k,p^{\alpha k})|v(p^k)|<\infty,\;\sum\limits_{\ell=m}^{\infty}|v(p^{\ell})|<\infty,
\end{equation}
if $\alpha\neq1$, and
\begin{equation}\Label{3-3}\
\sum\limits_{k=-\infty}^m|k|\;p^k|v(p^k)|<\infty,\;\sum\limits_{\ell=m}^{\infty}\ell|v(p^{\ell})|<\infty,
\end{equation}
if $\alpha=1$. Then there exists $(D^{\alpha}I^{\alpha}v)(|x|_p)=v(|x|_p)$ for any $x\neq0$.
\end{lemma}

For an operator-theoretic treatment of the above operators, especially for $\alpha =1$, see \cite{K3}.

\section{Main results}

Let $\alpha>0$, $\gamma>0$. We consider the problem
\begin{equation}\Label{0-1}\
|t|_p^{\gamma}(D^{\alpha}u)(|t|_p)=f(|t|_p,u(|t|_p)),\;0\neq t\in\mathbb{Q}_p,
\end{equation}
with the initial condition
\begin{equation}\Label{0-2}\
u(0)=u_0.
\end{equation}

We suppose that the function $f:p^{\mathbb{Z}}\times\mathbb{R}\to\mathbb{R}$ satisfies the conditions
\begin{equation}\Label{0-5}\
|f(|t|_p,x)|\leq M,
\end{equation}
\begin{equation}\Label{0-6}\
|f(|t|_p,x)-f(|t|_p,y)|\leq F|x-y|,
\end{equation}
for all $t\in\mathbb{Q}_p,\;x,y\in\mathbb{R}$ and some constants $M,F$ independent on $t,x,y$.

With the problem \eqref{0-1}-\eqref{0-2} we associate the integral equation
\begin{equation}\Label{0-8}\
u(|t|_p)=u_0+I^{\alpha}\left[|\cdot|_p^{-\gamma}f(|\cdot|_p,u(|\cdot|_p))\right](|t|_p).
\end{equation}

\begin{definition}
We call a solution $u$ to the equations \eqref{0-8}, if it exists, a \textit{mild solution} to the Cauchy problem \eqref{0-1}-\eqref{0-2}.
\end{definition}

Further we will need the following Lemma.

\begin{lemma}\label{pr1}
Let
$$I_{\alpha,\sigma}=\int\limits_{|y|_p<|t|_p}\bigg\vert|t|_p^{\alpha-1}-|y|_p^{\alpha-1}\bigg\vert|y|_p^{\alpha \sigma}dy,\;\alpha\neq1,$$
$$I_{1,\sigma}=\int\limits_{|y|_p<|t|_p}\bigg\vert\log|t|_p-\log|y|_p\bigg\vert|y|_p^{\sigma}dy.$$
Then for any $\sigma>\max(-\frac{1}{\alpha},-1)$ (not necessarily integer):
$$I_{\alpha,\sigma}=d_{\alpha,\sigma}|t|_p^{\alpha(\sigma+1)},$$
with $d_{\alpha,\sigma}$ independent on $t$.
Moreover, for any $\varepsilon>0$ there exists a constant $A>0$, independent on $\sigma$, such that
\begin{equation}\Label{2-6}\
0<d_{\alpha,\sigma}\leq Aq^{-\alpha \sigma},\;\textnormal{for any}\;\sigma\geq\max(-\frac{1}{\alpha},-1)+\varepsilon.
\end{equation}
\end{lemma}

\begin{proof}
Let $\alpha>1$. For $|t|_p=p^n$ due to \eqref{0-11} we have
\begin{multline}\nonumber
I_{\alpha,\sigma}=|t|_p^{\alpha-1}\int\limits_{|y|_p\leq p^{n-1}}|y|_p^{\alpha \sigma}dy-\int\limits_{|y|_p\leq p^{n-1}}|y|_p^{\alpha(\sigma+1)-1}dy=\\
\nonumber
=p^{n(\alpha-1)}\frac{1-p^{-1}}{1-p^{-\alpha \sigma-1}}p^{(\alpha \sigma+1)(n-1)}-\frac{1-p^{-1}}{1-p^{-\alpha(\sigma+1)}}p^{\alpha(\sigma+1)(n-1)}=d_{\alpha,\sigma}p^{n\alpha(\sigma+1)},
\end{multline}
where the first integral converges under the condition $\alpha \sigma+1>0$, that is $\sigma>-\frac{1}{\alpha}$, and the second one - under the condition $\alpha(\sigma+1)>0$, that is $\sigma>-1$, and
\begin{multline}\nonumber
d_{\alpha,\sigma}=\frac{1-p^{-1}}{1-p^{-\alpha \sigma-1}}p^{-\alpha \sigma-1}-\frac{1-p^{-1}}{1-p^{-\alpha \sigma-\alpha}}p^{-\alpha \sigma-\alpha}=\\
\nonumber
=(1-p^{-1})\frac{p^{\alpha-1}-1}{(1-p^{-\alpha \sigma-1})(p^{\alpha}-p^{-\alpha \sigma})}p^{-\alpha \sigma}.
\end{multline}
If for some $\varepsilon>0$ we have $\sigma\geq\max(-\frac{1}{\alpha},-1)+\varepsilon\eqqcolon\theta$, then
$$d_{\alpha,\sigma}\leq (1-p^{-1})\frac{p^{\alpha-1}-1}{(1-p^{-\alpha \theta-1})(p^{\alpha}-p^{-\alpha \theta})}p^{-\alpha \sigma}\eqqcolon Ap^{-\alpha \sigma}.$$

Similarly, for $0<\alpha<1$ we have
$$I_{\alpha,\sigma}=\int\limits_{|y|_p\leq p^{n-1}}|y|_p^{\alpha(\sigma+1)-1}dy-|t|_p^{\alpha-1}\int\limits_{|y|_p\leq p^{n-1}}|y|_p^{\alpha \sigma}dy=d_{\alpha,\sigma}|t|_p^{\alpha(\sigma+1)},$$
where
$$d_{\alpha,\sigma}=(1-p^{-1})\frac{1-p^{\alpha-1}}{(1-p^{-\alpha \sigma-1})(p^{\alpha}-p^{-\alpha \sigma})}p^{-\alpha \sigma}\leq Ap^{-\alpha \sigma},\;\sigma\geq\max(-\frac{1}{\alpha},-1)+\varepsilon.$$

If $\alpha=1$, we have
\begin{multline}\nonumber
I_{1,\sigma}=\sum\limits_{k=-\infty}^{n-1}\int\limits_{|y|_p=p^k}\left(\log|t|_p-\log|y|_p\right)|y|_p^{\sigma}dy=\\
\nonumber
=\sum\limits_{k=-\infty}^{n-1}\bigg\{n\log p\int\limits_{|y|_p=p^k}q^{k\sigma}dy-\int\limits_{|y|_p=p^k}k(\log p)p^{k\sigma}dy\bigg\}=\\
\nonumber
\\=(1-p^{-1})\log p\sum\limits_{k=-\infty}^{n-1}(n-k)p^{k(\sigma+1)}=(1-p^{-1})\log p\sum\limits_{\nu=1}^{\infty}\nu p^{(n-k)(\sigma+1)}=\\
\nonumber
=d_{1,\sigma}p^{n\alpha(\sigma+1)},
\end{multline}
where for $\sigma>-1$,
$$d_{1,\sigma}=(1-p^{-1})\log p\sum\limits_{\nu=1}^{\infty}\nu p^{-\nu(\sigma+1)}=(1-p^{-1})\log p\frac{p^{-\sigma-1}}{(1-p^{-\sigma-1})^2},$$
and for $\sigma\geq-1+\varepsilon\eqqcolon\theta$ we have
$$d_{1,\sigma}\leq(1-p^{-1})\log p\frac{p^{-1}}{(1-p^{-\theta-1})^2}p^{-\sigma}\eqqcolon Ap^{-\sigma}.$$

\end{proof}

The next three propositions contain estimates of $I^\alpha \varphi$ for a radial function $\varphi$.

\medskip
\begin{proposition}\label{pr2}
Let $\gamma<\min(1,\alpha)$ and let $\varphi=\varphi(|t|_p)$ be a radial function such that for all $0\neq t\in\mathbb{Q}_p$:
\begin{equation}\Label{0-8.4}
\abs*{\varphi(|t|_p)}\leq\mu|t|_p^{-\gamma}.
\end{equation}
with $\mu >0$.  Then the function $\varphi$ satisfies the conditions \eqref{2-7}-\eqref{2-8}, $I^{\alpha}\varphi$ exists, is a radial function, and there exists a constant $C_0$ such that
\begin{equation}\Label{0-8.5}\
\abs*{(I^{\alpha}\varphi)(|t|_p)}\leq C_0\mu|t|_p^{\alpha-\gamma}.
\end{equation}
\end{proposition}

\begin{proof}
Let us check the conditions of Lemma \ref{lemma2} for the function $\varphi$. Let $\alpha>1$. For any $m\leq0$ we have
$$\sum\limits_{k=-\infty}^m\max\, (p^k,p^{\alpha k})\,|\varphi(p^k)|=\sum\limits_{k=-\infty}^mp^k|\varphi(p^k)|\leq\mu\sum\limits_{k=-\infty}^mp^{(1-\gamma)k}<+\infty,$$
since $\gamma<1$. In the case $0<\alpha<1$ we have
$$\sum\limits_{k=-\infty}^m\max\,(p^k,p^{\alpha k})|\varphi(p^k)|=\sum\limits_{k=-\infty}^mp^{\alpha k}|\varphi(p^k)|\leq\mu\sum\limits_{k=-\infty}^mp^{(\alpha-\gamma)k}<+\infty,$$
since $\gamma<\alpha$. If $\alpha=1$, we have
$$\sum\limits_{k=-\infty}^m|k|\,p^k|\varphi(p^k)|\leq\mu\sum\limits_{k=-\infty}^m|k|p^{(1-\gamma)k}<+\infty.$$
Thus it follows from Lemma \ref{lemma2} that $I^{\alpha}\varphi$ exists. Moreover, for $\alpha\neq1$ we have:
\begin{multline}\nonumber
\abs*{(I^{\alpha}\varphi)(|t|_p)}=\abs*{p^{-\alpha}|t|_p^{\alpha}\varphi(|t|_p)+\frac{1-p^{-\alpha}}{1-p^{\alpha-1}}\int\limits_{|y|_p<|t|_p}\left(|t|_p^{\alpha-1}-|y|_p^{\alpha-1}\right)\varphi(|y|_p)dy}\leq\\
\nonumber
\leq p^{-\alpha}\mu|t|_p^{\alpha-\gamma}+\frac{1-p^{-\alpha}}{1-p^{\alpha-1}}\mu\int\limits_{|y|_p<|t|_p}\abs*{|t|_p^{\alpha-1}-|y|_p^{\alpha-1}}|y|_p^{-\gamma}dy=\\
\nonumber
=p^{-\alpha}\mu|t|_p^{\alpha-\gamma}+\frac{1-p^{-\alpha}}{1-p^{\alpha-1}}\mu I_{\alpha,-\frac{\gamma}{\alpha}}.
\end{multline}
Now we apply Lemma \ref{pr1} with $\sigma=-\frac{\gamma}{\alpha}$, then for $\gamma<\min(1,\alpha)$ we have $\sigma>\max(-\frac{1}{\alpha},-1)$. Therefore
$$\abs*{(I^{\alpha}\varphi)(|t|_p)}\leq C_0\mu|t|_p^{\alpha-\gamma},$$
where $C_0=p^{-\alpha}+\frac{1-p^{-\alpha}}{1-p^{\alpha-1}}d_{\alpha,-\frac{\gamma}{\alpha}}$.

For $\alpha=1$ we have
\begin{multline}\nonumber
\abs*{(I^1\varphi)(|t|_p)}=\abs*{p^{-1}|t|_p\varphi(|t|_p)+\frac{1-p}{p\log p}\int\limits_{|y|_p<|t|_p}\left(\log|t|_p-\log|y|_p\right)\varphi(|y|_p)dy}\leq\\
\nonumber
\leq p^{-1}\mu|t|_p^{1-\gamma}+\frac{1-p}{p\log p}\mu I_{1,-\gamma}.
\end{multline}
For $\gamma<1$ we can apply Lemma \ref{pr1} with $\sigma=-\gamma$ to obtain
$$\abs*{(I^1\varphi)(|t|_p)}\leq C_0\mu|t|_p^{1-\gamma},$$
where $C_0=p^{-1}+\frac{1-p}{p\log p}d_{1,-\gamma}$, which finishes the proof.
\end{proof}

\begin{proposition}\label{pr3}
Let $n\in\mathbb{N}$, $\gamma<\min(1,\alpha)$ and let $\varphi=\varphi(|t|_p)$ be a radial function such that for all $0\neq t\in\mathbb{Q}_p$:
\begin{equation}\Label{0-8.50}
\abs*{\varphi(|t|_p)}\leq\mu|t|_p^{n\alpha-(n+1)\gamma}.
\end{equation}
Then $I^{\alpha}\varphi$ exists, is a radial function, and there exists a constant $C_n$ such that
\begin{equation}\Label{0-8.51}
\abs*{(I^{\alpha}\varphi)(|t|_p)}\leq C_n\mu|t|_p^{(n+1)(\alpha-\gamma)}.
\end{equation}
\end{proposition}

\begin{proof}
For $\alpha>1$ we have
$$\sum\limits_{k=-\infty}^m\max\,(p^k,p^{\alpha k})|\varphi(p^k)|\leq\mu\sum\limits_{k=-\infty}^mp^kp^{[n\alpha-(n+1)\gamma]k}<+\infty,$$
since $\gamma<1=\frac{n+1}{n+1}<\frac{n\alpha+1}{n+1}$, and hence $1+n\alpha-(n+1)\gamma>0$. If $0<\alpha<1$, we have
$$\sum\limits_{k=-\infty}^m\max\,(p^k,p^{\alpha k})|\varphi(p^k)|\leq\mu\sum\limits_{k=-\infty}^mp^{\alpha k}p^{[n\alpha-(n+1)\gamma]k}<+\infty,$$
since $\alpha>\gamma$, and thus $(n+1)\alpha-(n+1)\gamma>0$. For $\alpha=1$ we have
\begin{multline}\nonumber
\sum\limits_{k=-\infty}^m|k|p^k|\varphi(p^k)|\leq\mu\sum\limits_{k=-\infty}^m|k|p^kp^{[n\alpha-(n+1)\gamma]k}=\\
\nonumber
=\mu\sum\limits_{k=-\infty}^m|k|\,p^{[1+n-(n+1)\gamma]k}<+\infty,
\end{multline}
since $\gamma<1$.

Therefore the conditions \eqref{2-7}-\eqref{2-8} are satisfied, and by Lemma \ref{lemma2} $I^{\alpha}\varphi$ exists. For $\alpha\neq1$ we have
\begin{multline}\nonumber
\abs*{(I^{\alpha}\varphi)(|t|_p)}=\abs*{p^{-\alpha}|t|_p^{\alpha}\varphi(|t|_p)+\frac{1-p^{-\alpha}}{1-p^{\alpha-1}}\int\limits_{|y|_p<|t|_p}\left(|t|_p^{\alpha-1}-|y|_p^{\alpha-1}\right)\varphi(|y|_p)dy}\leq\\
\nonumber
\leq p^{-\alpha}\mu|t|_p^{\alpha}|t|_p^{n\alpha-(n+1)\gamma}+\frac{1-p^{-\alpha}}{1-p^{\alpha-1}}\mu\int\limits_{|y|_p<|t|_p}\abs*{|t|_p^{\alpha-1}-|y|_p^{\alpha-1}}|y|_p^{n\alpha-(n+1)\gamma}dy=\\
\nonumber
=p^{-\alpha}\mu|t|_p^{(n+1)(\alpha-\gamma)}+\frac{1-p^{-\alpha}}{1-p^{\alpha-1}}\mu I_{\alpha,\frac{n\alpha-(n+1)\gamma}{\alpha}}.
\end{multline}
To apply Lemma \ref{pr1} with $\sigma=\frac{n\alpha-(n+1)\gamma}{\alpha}$, let us verify that $\sigma>\max(-\frac{1}{\alpha},-1)$. Indeed, $(n+1)(\alpha-\gamma)>0$, so $n\alpha-(n+1)\gamma>-1$, that is $\frac{n\alpha-(n+1)\gamma}{\alpha}>-1$. If $\alpha>1$, we have $\gamma<1=\frac{n+1}{n+1}<\frac{n\alpha+1}{n+1}$, and if $0<\alpha<1$, then $\gamma<\alpha=\frac{n\alpha+\alpha}{n+1}<\frac{n\alpha+1}{n+1}$, hence in both cases $\frac{n\alpha-(n+1)\gamma}{\alpha}>-\frac{1}{\alpha}$. Therefore
$$\abs*{(I^{\alpha}\varphi)(|t|_p)}\leq C_n\mu|t|_p^{(n+1)(\alpha-\gamma)},$$

where $C_n=p^{-\alpha}+\frac{1-p^{-\alpha}}{1-p^{\alpha-1}}d_{\alpha,\frac{n\alpha-(n+1)\gamma}{\alpha}}$.

For $\alpha=1$ we have $(n+1)(1-\gamma)>0$, therefore $n-(n+1)\gamma>-1$, and for $\sigma=n-(n+1)\gamma$ by Lemma \ref{pr1} we have
\begin{multline}\nonumber
\abs*{(I^1\varphi)(|t|_p)}=\abs*{p^{-1}|t|_p\varphi(|t|_p)+\frac{1-p}{p\log p}\int\limits_{|y|_p<|t|_p}\left(\log|t|_p-\log|y|_p\right)\varphi(|y|_p)dy}\leq\\
\nonumber
\leq p^{-1}\mu|t|_p^{(n+1)(1-\gamma)}+\frac{1-p}{p\log p}\mu I_{1,n-(n+1)\gamma}=C_n\mu|t|_p^{1-\gamma},
\end{multline}
where 
\begin{equation}\Label{Cn}\ 
C_n=p^{-1}+\frac{1-p}{p\log p}d_{1,n-(n+1)\gamma}.
\end{equation}
\end{proof}

\begin{proposition}\label{pr4}
Let $\gamma<\min(1,\alpha)$. Then the sequence of constants $\{C_n:\;n\geq1\}$ defined in \eqref{Cn} is bounded.
\end{proposition}

\begin{proof}
Since for all $n\geq 0$ we have $\frac{n\alpha-(n+1)\gamma}{\alpha}>\max(-\frac{1}{\alpha},-1)$, and also $\frac{n\alpha-(n+1)\gamma}{\alpha}=\frac{n(\alpha-\gamma)-\gamma}{\alpha}\geq-\frac{\gamma}{\alpha}>\max(-\frac{1}{\alpha},-1)$, then for $\alpha\neq1$ we can apply Lemma \ref{pr1}, then from \eqref{2-6} for $\sigma=\frac{n\alpha-(n+1)\gamma}{\alpha}$ we obtain for any $n\geq0$
\begin{multline}\nonumber
C_n=p^{-\alpha}+\frac{1-p^{-\alpha}}{1-p^{\alpha-1}}d_{\alpha,\frac{n\alpha-(n+1)\gamma}{\alpha}}\leq p^{-\alpha}+\frac{1-p^{-\alpha}}{1-p^{\alpha-1}}Ap^{n(\gamma-\alpha)+\gamma}\leq\\
\nonumber
\leq p^{-\alpha}+\frac{1-p^{-\alpha}}{1-p^{\alpha-1}}Ap^{\gamma}\eqqcolon C
\end{multline}
and for $\alpha=1$:
\begin{multline}\nonumber
C_n=p^{-1}+\frac{1-p}{p\log p}d_{1,n-(n+1)\gamma}\leq p^{-1}+\frac{1-p}{p\log p}Ap^{n(\gamma-1)+\gamma}\leq\\
\nonumber
\leq p^{-1}+\frac{1-p}{p\log p}Ap^{\gamma}\eqqcolon C.
\end{multline}
\end{proof}

Let us denote
\begin{equation}\Label{00-8}
\widetilde{f}(|t|_p,u(|t|_p))=|t|_p^{-\gamma}f(|t|_p,u(|t|_p)),\;0\neq t\in\mathbb{Q}_p.
\end{equation}
Then the function $\widetilde{f}$ satisfies:
\begin{equation}\Label{00-5}\
|\widetilde{f}(|t|_p,x)|\leq M|t|_p^{-\gamma},\mbox{ for all }t\in\mathbb{Q}_p,\;x\in\mathbb{R},
\end{equation}
\begin{equation}\Label{00-6}\
|\widetilde{f}(|t|_p,x)-\widetilde{f}(|t|_p,y)|\leq F|x-y||t|_p^{-\gamma},\mbox{ for all }t\in\mathbb{Q}_p,\;x,y\in\mathbb{R},
\end{equation}
and the problem \eqref{0-1} is equivalent to 
\begin{equation}\Label{0-3}\
(D^{\alpha}u)(|t|_p)=\widetilde{f}(|t|_p,u(|t|_p)),\;0\neq t\in\mathbb{Q}_p,
\end{equation}

The next result deals with the existence and uniqueness of local mild solutions.

\begin{theorem}\label{th1}
Suppose that the conditions \eqref{0-5}, \eqref{0-6} and $\gamma<\min(1,\alpha)$ hold. Then the problem \eqref{0-1}-\eqref{0-2} has a unique local mild solution, that is the integral equation \eqref{0-8} has a solution $u(|t|_p)$ defined for $|t|_p\leq p^N$, where $N\in\mathbb{Z}$ is sufficiently negative, and any other solution $\overline{u}(|t|_p)$, if it exists, coincides with $u$ for $|t|_p\leq p^K$, where $K\leq N$.
\end{theorem}

\begin{proof}
We look for a solution of the equation \eqref{0-8} as a limit of the sequence $\{u_k\}$, where the initial value $u_0$ is as in \eqref{0-8}:
\begin{equation}\Label{0-9}\
u_k(|t|_p)=u_0+I^{\alpha}\widetilde{f}(|\cdot|_p,u_{k-1}(|\cdot|_p))(|t|_p).
\end{equation}

We will prove by induction that for all $k\geq0$:
\begin{equation}\Label{0-14}\
\abs*{u_{k+1}(|t|_p)-u_k(|t|_p)}\leq C^{k+1}MF^k|t|_p^{(k+1)(\alpha-\gamma)}.
\end{equation}

Indeed, it follows from the condition \eqref{00-5}, Lemma \ref{pr2} with $\varphi=\widetilde{f},\;\mu=M$ and Proposition \ref{pr4} that
\begin{equation}\Label{0-13}\
\abs*{u_1(|t|_p)-u_0}=\abs*{I^{\alpha}\widetilde{f}(|\cdot|_p,u_0)(|t|_p)}\leq C_0M|t|_p^{\alpha-\gamma}\leq CM|t|_p^{\alpha-\gamma}.
\end{equation}
Hence
$$\abs*{u_2(|t|_p)-u_1(|t|_p)}=I^{\alpha}\left[\widetilde{f}(|\cdot|_p,u_1(|\cdot|_p))-\widetilde{f}(|\cdot|_p,u_0)\right](|t|_p),$$
where we used that \eqref{00-6} and \eqref{0-13} imply
$$\abs*{\widetilde{f}(|t|_p,u_1(|t|_p))-\widetilde{f}(|t|_p,u_0)}\leq F|t|_p^{-\gamma}\abs*{u_1(|t|_p)-u_0}\leq CMF|t|_p^{\alpha-2\gamma}.$$
Then it follows from Proposition \ref{pr3} for $n=1,\;\varphi=\widetilde{f}(|\cdot|_p,u_1(|\cdot|_p))-\widetilde{f}(|\cdot|_p,u_0)$, $\mu=CMF$ that
\begin{multline}\nonumber
\abs{u_2(|t|_p)-u_1(|t|_p)}=\abs*{I^{\alpha}\left[\widetilde{f}(|\cdot|_p,u_1(|\cdot|_p))-\widetilde{f}(|\cdot|_p,u_0)\right]}\leq\\
\nonumber
\leq C_1CMF|t|_p^{2(\alpha-\gamma)}\leq C^2MF|t|_p^{2(\alpha-\gamma)}.
\end{multline}

Suppose that \eqref{0-14} holds for $\abs*{u_{k+1}(|t|_p)-u_k(|t|_p)}$. Then
$$\abs*{u_{k+2}(|t|_p)-u_{k+1}(|t|_p)}=\abs*{I^{\alpha}\left[\widetilde{f}(|\cdot|_p,u_{k+1}(|\cdot|_p))-\widetilde{f}(|\cdot|_p,u_k(|\cdot|_p))\right](|t|_p)},$$
and it follows from \eqref{00-6} and the induction hypothesis that
\begin{multline}\nonumber
\abs*{\widetilde{f}(|t|_p,u_{k+1}(|t|_p))-\widetilde{f}(|t|_p,u_k(|t|_p))}\leq F|t|_p^{-\gamma}\abs*{u_{k+1}(|t|_p)-u_k(|t|_p)}\leq\\
\nonumber
\leq C^{k+1}MF^{k+1}|t|_p^{(k+1)(\alpha-\gamma)}|t|_p^{-\gamma}=C^{k+1}MF^{k+1}|t|_p^{(k+1)\alpha-(k+2)\gamma}.
\end{multline}
Hence by Proposition \ref{pr3}, applied to the function 
$$\varphi=\widetilde{f}(|\cdot|_p,u_{k+1}(|\cdot|_p))-\widetilde{f}(|\cdot|_p,u_k(|\cdot|_p)),$$
 we have
$$\abs*{u_{k+2}(|t|_p)-u_{k+1}(|t|_p)}\leq C^{k+2}MF^{k+1}|t|_p^{(k+2)(\alpha-\gamma)}.$$

Therefore the sequence $\{u_k\}$ converges uniformly on the ball $|t|_p\leq T$, for sufficiently small $T$, to the limit which we denote by $u(|t|_p)$. It follows from \eqref{2-7}, \eqref{2-8} that
$$I^{\alpha}\widetilde{f}(|\cdot|_p,u_k(|\cdot|_p))(|t|_p)\to I^{\alpha}\widetilde{f}(|\cdot|_p,u(|\cdot|_p))(|t|_p),\;|t|_p\leq T,$$
so that $u$ is a local solution to the equation \eqref{0-8}.

Let us prove the uniqueness of this solution. Suppose that there exists another local solution $\overline{u}(|t|_p)$. Then by Proposition \ref{pr2} we have
$$|u(|t|_p)-\overline{u}(|t|_p)|\leq2MC|t|_p^{\alpha-\gamma}.$$
On the other hand,
\begin{multline}\nonumber
u(|t|_p)-\overline{u}(|t|_p)=I^{\alpha}\widetilde{f}(|\cdot|_p,u(|\cdot|_p))(|t|_p)-I^{\alpha}\widetilde{f}(|\cdot|_p,\overline{u}(|\cdot|_p))(|t|_p)=\\
\nonumber
=I^{\alpha}\left\{\widetilde{f}(|\cdot|_p,u(|\cdot|_p))-\widetilde{f}(|\cdot|_p,\overline{u}(|\cdot|_p))\right\}(|t|_p),
\end{multline}
and it follows from \eqref{00-6} and Proposition \ref{pr3} that
$$|u(|t|_p)-\overline{u}(|t|_p)|\leq2MC^2F|t|_pI^{2(\alpha-\gamma)}.$$
We show by induction that
$$|u(|t|_p)-\overline{u}(|t|_p)|\leq2MC^mF^{m-1}|t|_p^{m(\alpha-\gamma)}.$$
Suppose that the statement holds for some $m$. We have by \eqref{00-6}
\begin{multline}\nonumber
\abs*{\widetilde{f}(|t|_p,u(|t|_p)-\widetilde{f}(|t|_p,\overline{u}(|t|_p)}\leq\\
\nonumber 
\leq|t|_p^{-\gamma}F\big|u (|t|_p)-\overline{u}(|t|_p)\big|\leq 2F^mMC^m|t|_p^{m\alpha-(m+1)\gamma},
\end{multline}
hence it follows from Proposition \ref{pr3} that
$$|u(|t|_p)-\overline{u}(|t|_p)|\leq 2MC^{m+1}F^m|t|_p^{(m+1)(\alpha-\gamma)}.$$

For $|t|_p<1$, by letting $m\to\infty$, we obtain that $u(|t|_p)=\overline{u}(|t|_p)$.
\end{proof}

\section{Extension of solutions}

Suppose that $\gamma<\min(1,\alpha)$, the function $f$ satisfies the conditions \eqref{0-5} and \eqref{0-6}, and $\widetilde{f}$ is defined in \eqref{00-8}. Then by Theorem \ref{th1} we can construct a local solution $u(|t|_p),\;|t|_p\leq p^N,\;N\in\mathbb{Z}$, to the integral equation \eqref{0-8}. Let $\alpha\neq1$. We can apply Lemma \ref{lemma2} to the function $\widetilde{f}$ and substitute the explicit expression for $I^{\alpha}\widetilde{f}$ into the equation \eqref{0-8}. Therefore, in order to find a solution for $|t|_p=p^{N+1}$, we have to solve the equation
\begin{equation}\Label{0-15}\
u(p^{N+1})=u_0+v_0^{(N)}+p^{\alpha N}\widetilde{f}(p^{N+1},u(p^{N+1})),
\end{equation}
where
\begin{equation}\Label{0-16}\
v_0^{(N)}=\frac{1-p^{-\alpha}}{1-p^{\alpha-1}}\int\limits_{|y|_p\leq p^N}\left(p^{(N+1)(\alpha-1)}-|y|_p^{\alpha-1}\right)\widetilde{f}(|y|_p,u(|y|_p))dy
\end{equation}
is a known constant.

Likewise, for $\alpha=1$:
\begin{equation}\Label{0-151}\
u(p^{N+1})=u_0+v_0^{(N)}+p^N\widetilde{f}(p^{N+1},u(p^{N+1})),
\end{equation}
\begin{equation}\Label{0-161}\
v_0^{(N)}=\frac{1-p}{p\log p}\int\limits_{|y|_p\leq p^N}\left(\log p^{N+1}-\log|y|_p\right)\widetilde{f}(|y|_p,u(|y|_p))dy.
\end{equation}

\begin{theorem}\label{th2}
Suppose that $\gamma<\min(1,\alpha)$, the function $f$ satisfies the conditions \eqref{0-5} and \eqref{0-6} in the following form:
\begin{equation}\Label{0-19}\
|f(p^{\ell},x)-f(p^{\ell},y)|\leq F_{\ell}|x-y|,\;x,y\in\mathbb{R},\;\ell\in\mathbb{Z},
\end{equation}
where $0<F_{\ell}<p^{-\alpha \ell}$ for all $\ell\in\mathbb{Z}$.
Then the local solution to \eqref{0-8} can be extended to a global solution defined for all $t\in\mathbb{Q}_p$.
\end{theorem}

\begin{proof}
For all $\ell\geq0$ we have
$$p^{\alpha \ell}\abs*{\widetilde{f}(p^{\ell+1},x)-\widetilde{f}(p^{\ell+1},y)}\leq p^{\alpha \ell}p^{-\alpha(\ell+1)}p^{\gamma}|x-y|=p^{\gamma-\alpha}|x-y|,$$
where $p^{\gamma-\alpha}<1$, since $\gamma<\alpha$, and hence the operator $u_0+v_0^{(\ell)}+p^{\alpha \ell}\widetilde{f}(p^{\ell+1},\cdot)$ is a contraction mapping in the Banach space of continuous functions on the ball $\{|t|_p\leq p^{\ell+1}\}$, and it follows from the Banach fixed point theorem that it has a unique fixed point $u(p^{\ell+1})$:
$$u(p^{\ell+1})=u_0+v_0^{(\ell)}+p^{\alpha \ell}\widetilde{f}(p^{\ell+1},u(p^{\ell+1})).$$
\end{proof}

\section{Transition from an integral equation to a differential one}

Now we may turn back to the differential form of the initial equation.

Suppose that $f$ satisfies the conditions \eqref{0-5} and \eqref{0-19}. Let $u(|t|_p)$ be a solution obtained by iterations for $|t|_p\leq p^N$ and extended as in \eqref{0-15}. Then for all $\ell\geq N$:
\begin{equation}\Label{0-23}\
u(p^{\ell+1})=u_0+v_0^{(\ell)}+p^{\alpha \ell}\widetilde{f}(p^{\ell+1},u(p^{\ell+1})),
\end{equation}
where $v_0^{(\ell)}$ is defined in \eqref{0-16}, \eqref{0-161}.

\begin{proposition}\label{pr5}
Suppose that the function $f$ satisfies the conditions \eqref{0-19} and
\begin{equation}\Label{0-21}\
|f(p^{\ell},x)|\leq Ap^{-\beta \ell},\;\ell\geq1,\;\forall x\in\mathbb{R},
\end{equation}
where $\beta+\gamma>\alpha$. Then the function $u(|t|_p),\;|t|_p\leq p^{\ell+1}$, which solves the equation \eqref{0-8}, satisfies the condition \eqref{3-1}, and the function $\widetilde{f}$ satisfies the conditions \eqref{3-2}-\eqref{3-3}.
\end{proposition}

\begin{proof}
Let $\alpha\neq1$.
For $\ell\geq1$ we denote $v_0^{(\ell)}=v_{0,1}^{(\ell)}+v_{0,2}^{(\ell)}$, where
$$v_{0,1}^{(\ell)}=\frac{1-p^{-\alpha}}{1-p^{\alpha-1}}\int\limits_{|y|_p\leq 1}\left(p^{(\ell+1)(\alpha-1)}-|y|_p^{\alpha-1}\right)\widetilde{f}(|y|_p,u(|y|_p))dy,$$
$$v_{0,2}^{(\ell)}=\frac{1-p^{-\alpha}}{1-p^{\alpha-1}}\int\limits_{p\leq|y|_p\leq p^{\ell}}\left(p^{(\ell+1)(\alpha-1)}-|y|_p^{\alpha-1}\right)\widetilde{f}(|y|_p,u(|y|_p))dy.$$

Then it follows from \eqref{0-5} that
\begin{multline}\nonumber
|v_{0,1}^{(\ell)}|\leq\abs*{\frac{1-p^{-\alpha}}{1-p^{\alpha-1}}}M\int\limits_{|y|_p\leq1}\abs*{p^{(\ell+1)(\alpha-1)}-|y|_p^{\alpha-1}}|y|_p^{-\gamma}dy\leq\\
\nonumber
\leq\abs*{\frac{1-p^{-\alpha}}{1-p^{\alpha-1}}}M\bigg\{p^{(\ell+1)(\alpha-1)}\int\limits_{|y|_p\leq1}|y|_p^{-\gamma}dy+\int\limits_{|y|_p\leq1}|y|_p^{\alpha-1-\gamma}dy\bigg\}=\\
\nonumber
=\abs*{\frac{1-p^{-\alpha}}{1-p^{\alpha-1}}}M\left\{p^{(\ell+1)(\alpha-1)}p\frac{1-p^{-1}}{1-p^{\gamma-1}}+\frac{1-p^{-1}}{1-p^{-\alpha+\gamma}}\right\}\leq C_1\left\{p^{(\ell+1)(\alpha-1)}+1\right\},\\
\end{multline}
where $C_1=\abs*{\frac{1-p^{-\alpha}}{1-p^{\alpha-1}}}M\max\left(\frac{1-p^{-1}}{1-p^{\gamma-1}},\frac{1-p^{-1}}{1-p^{-\alpha+\gamma}}\right)$.

By \eqref{0-21} we have
\begin{multline}\nonumber
|v_{0,2}^{(\ell)}|\leq\\
\nonumber
\leq\abs*{\frac{1-p^{-\alpha}}{1-p^{\alpha-1}}}A\bigg\{p^{(\ell+1)(\alpha-1)}\int\limits_{p\leq|y|_p\leq p^{\ell}}|y|_p^{-(\beta+\gamma)}dy+\int\limits_{p\leq|y|_p\leq p^{\ell}}|y|_p^{\alpha-1-(\beta+\gamma)}dy\bigg\}=\\
\nonumber
=\abs*{\frac{1-p^{-\alpha}}{1-p^{\alpha-1}}}A\bigg\{p^{(\ell+1)(\alpha-1)}\sum\limits_{j=1}^{\ell}\int\limits_{|y|_p=p^j}|y|_p^{-(\beta+\gamma)}dy+\sum\limits_{j=1}^{\ell}\int\limits_{|y|_p=q^j}|y|_p^{\alpha-1-(\beta+\gamma)}dy\bigg\}=\\
\nonumber
=\abs*{\frac{1-p^{-\alpha}}{1-p^{\alpha-1}}}A\bigg\{p^{(\ell+1)(\alpha-1)}\sum\limits_{j=1}^{\ell}p^{-j(\beta+\gamma)}(1-\frac{1}{p})p^j+\sum\limits_{j=1}^{\ell}p^{j(\alpha-1-\beta-\gamma)}(1-\frac{1}{p})p^j\bigg\}=\\
\nonumber
=C_2p^{(\ell+1)(\alpha-1)}\sum\limits_{j=1}^{\ell}p^{(1-\beta-\gamma)j}+C_2\sum\limits_{j=1}^{\ell}p^{(\alpha-\beta-\gamma)j},
\end{multline}
where $C_2=\abs*{\frac{1-p^{-\alpha}}{1-p^{\alpha-1}}}A(1-\frac{1}{p})$.

Suppose additionally that $\alpha>1$. Then, due to assumption of the Theorem, $\beta+\gamma>\alpha>1$ and we have

\begin{equation}\Label{5-3}
p^{(\ell+1)(\alpha-1)}\sum\limits_{j=1}^{\ell}p^{(1-\beta-\gamma)j}=\frac{p^{(\ell+1)(\alpha-1)}p^{1-\beta-\gamma}\left(1-p^{\ell(1-\beta-\gamma)}\right)}{1-p^{1-\beta-\gamma}}\leq C_3p^{(\ell+1)(\alpha-1)},
\end{equation}
where $C_3=\frac{p^{1-\beta-\gamma}}{1-p^{1-\beta-\gamma}}$.

Let $0<\alpha<1$. If $\beta+\gamma>1$, then similarly to \eqref{5-3},
$$p^{(\ell+1)(\alpha-1)}\sum\limits_{j=1}^{\ell}p^{(1-\beta-\gamma)j}\leq C_3p^{(\ell+1)(\alpha-1)},$$
and if $\beta+\gamma\leq1$, then
\begin{multline}\nonumber
p^{(\ell+1)(\alpha-1)}\sum\limits_{j=1}^{\ell}p^{(1-\beta-\gamma)j}=p^{(\ell+1)(\alpha-1)}\frac{p^{1-\beta-\gamma}\left(p^{\ell(1-\beta-\gamma)}-1\right)}{p^{1-\beta-\gamma}-1}=\\
\nonumber
=C_4\left[p^{(\ell+1)(\alpha-\beta-\gamma)}-p^{(\ell+1)(\alpha-1)+1-\beta-\gamma}\right]\leq C_4p^{(\ell+1)(\alpha-\beta-\gamma)},
\end{multline}
where $C_4=\frac{1}{p^{1-\beta-\gamma}-1}$. Moreover, for all $\alpha\neq1$ we have
$$\sum\limits_{j=1}^{\ell}p^{(\alpha-\beta-\gamma)j}=\frac{p^{\alpha-\beta-\gamma}\left(1-p^{\ell(\alpha-\beta-\gamma)}\right)}{1-p^{\alpha-\beta-\gamma}}\leq \frac{p^{\alpha-\beta-\gamma}}{1-p^{\alpha-\beta-\gamma}}=C_5.$$

If $\alpha=1$, we have for $\ell\geq1$:
$$u(p^{\ell+1})=u_0+v_0^{(\ell)}+p^{\ell}\widetilde{f}(p^{\ell+1},u(p^{\ell+1})),$$
it follows from \eqref{0-5} that
\begin{multline}\nonumber
|v_{0,1}^{(\ell)}|=\abs*{\frac{1-p}{p\log p}\int\limits_{|y|_p\leq1}\left(\log p^{\ell+1}-\log|y|_p\right)\widetilde{f}(|y|_p,u(|y|_p))dy}\leq\\
\nonumber
\leq\frac{p-1}{p\log p}M\int\limits_{|y|_p\leq1}|y|_p^{-\gamma}\abs*{\log p^{\ell+1}-\log|y|_p}dy\leq\\
\nonumber
\leq\frac{p-1}{p\log p}M\bigg\{\log p^{\ell+1}\int\limits_{|y|_p\leq1}|y|_p^{-\gamma}dy-\sum\limits_{n=-\infty}^0\int\limits_{|y|_p=p^n}|y|_p^{-\gamma}\log|y|_pdy\bigg\}=\\
\nonumber
=\frac{p-1}{p\log p}M\bigg\{\log p^{\ell+1}\frac{1-p^{-1}}{1-p^{\gamma-1}}-\sum\limits_{n=-\infty}^0p^{-\gamma n}(1-p^{-1})p^n\log p^n\bigg\}=\\
\nonumber
=\frac{p-1}{p\log p}\log p(1-p^{-1})M\bigg\{\frac{\ell+1}{1-p^{\gamma-1}}+\frac{p^{\gamma-1}}{(1-p^{\gamma-1})^2}\bigg\}\leq C_7(\ell+1),
\end{multline}
where the sum $\sum\limits_{n=0}^{\infty}np^{-(1-\gamma)n}$ is calculated due to the identity 0.231.2 in \cite{Grad}. Further,

\begin{multline}\nonumber
|v_{0,2}^{(\ell)}|=\abs*{\frac{1-p}{p\log p}\int\limits_{p\leq|y|_p\leq p^{\ell}}\left(\log p^{\ell+1}-\log|y|_p\right)\widetilde{f}(|y|_p,u(|y|_p))dy}\leq\\
\nonumber
\leq\frac{p-1}{p\log p}M\bigg\{\sum\limits_{j=1}^l\int\limits_{|y|_p=p^j}|y|_p^{-\gamma}\log p^{\ell+1}dy+\sum\limits_{j=1}^{\ell}|y|_p^{-\gamma}\log|y|_p\bigg\}\\
\nonumber
\leq\frac{p-1}{p\log p}M\log p\bigg\{(\ell+1)(1-p^{-1})\sum\limits_{j=1}^{\ell}p^{(1-\gamma)j}+\sum\limits_{j=1}^{\ell}jp^{(1-\gamma)j}\bigg\}\leq\\
\nonumber
\leq C_8(\ell+1)p^{l(1-\gamma)}.
\end{multline}

Suppose that the solution $u$ is obtained in the ball $B_N$ by iterations as in Theorem \ref{th1}. Then for all $m\geq N$ we have
$$\sum\limits_{k=-\infty}^mp^k|u(p^k)|\leq\max\limits_{t\in B_N}|u(|t|_p)|\sum\limits_{k=-\infty}^Np^k+\sum\limits_{k=N+1}^mp^k|u(p^k)|<\infty,$$
since $u$ is continuous on $B_N$. Let us check the condition \eqref{3-1'} for every term in \eqref{0-23}. Let $\alpha\neq1$. Then for all $m\geq 1$ we have
\begin{multline}\nonumber
\sum\limits_{\ell=m}^{\infty}p^{-\alpha \ell}|u_0(p^{\ell})|=|u_0|\sum\limits_{\ell=m}^{\infty}p^{-\alpha \ell}<\infty,\\
\nonumber
\sum\limits_{\ell=m}^{\infty}p^{-\alpha \ell}|v_{0,1}^{(\ell)}|\leq C_1\sum\limits_{\ell=m}^{\infty}p^{-\alpha \ell+(\ell+1)(\alpha-1)}+C_1\sum\limits_{\ell=m}^{\infty}p^{-\alpha \ell}=C_1p^{\alpha-1}\sum\limits_{\ell=m}^{\infty}p^{-\ell}+\\
\nonumber
+C_1\sum\limits_{\ell=m}^{\infty}p^{-\alpha \ell}<\infty.
\end{multline}
Let $\alpha>1$. Then
$$\sum\limits_{\ell=m}^{\infty}p^{-\alpha \ell}|v_{0,2}^{(\ell)}|\leq C_2C_3\sum\limits_{\ell=m}^{\infty}p^{-\alpha \ell}p^{(\ell+1)(\alpha-1)}+C_2C_5\sum\limits_{\ell=m}^{\infty}p^{-\alpha \ell}<\infty.$$
If $0<\alpha<1$ and additionally $\beta+\gamma>1$, then the similar estimate holds, and if $\beta+\gamma\leq1$, then
\begin{multline}\nonumber
\sum\limits_{\ell=m}^{\infty}p^{-\alpha \ell}|v_{0,2}^{(\ell)}|\leq C_2C_4\sum\limits_{\ell=m}^{\infty}p^{-\alpha \ell}p^{(\ell+1)(\alpha-\beta-\gamma)}+C_2C_5\sum\limits_{\ell=m}^{\infty}p^{-\alpha \ell}=\\
\nonumber
=C_2C_4p^{\alpha-\beta-\gamma}\sum\limits_{\ell=m}^{\infty}p^{-(\beta+\gamma)\ell}+C_2C_5\sum\limits_{\ell=m}^{\infty}p^{-\alpha \ell}<\infty.
\end{multline}
If $\alpha=1$, then
$$\sum\limits_{\ell=m}^{\infty}p^{-\alpha \ell}|v_{0,1}^{(\ell)}|\leq C_7\sum\limits_{\ell=m}^{\infty}(\ell+1)p^{-\ell}<\infty,$$
$$\sum\limits_{\ell=m}^{\infty}p^{-\alpha \ell}|v_{0,2}^{(\ell)}|\leq C_8\sum\limits_{\ell=m}^{\infty}(\ell+1)p^{-\alpha \ell}p^{(1-\gamma)\ell}=C_8\sum\limits_{\ell=m}^{\infty}(\ell+1)p^{-\gamma \ell}<\infty.$$

It follows from \eqref{0-21} that
$$\sum\limits_{\ell=m}^{\infty}p^{-\alpha \ell}|f(p^{\ell+1},u(p^{\ell+1}))|\leq A\sum\limits_{\ell=m}^{\infty}p^{(\alpha-(\beta+\gamma))(\ell+1)}<\infty.$$
Hence $u$ satisfies the condition \eqref{3-1}.

Consider now the function $\widetilde{f}$. By Proposition \ref{pr2}, noting \eqref{00-5}, \eqref{00-6}, $\widetilde{f}$ satisfies the first conditions in \eqref{3-2}-\eqref{3-3}. If $\alpha\neq1$, for $m>0$ we have
$$\sum\limits_{\ell=m}^{\infty}|\widetilde{f}(p^{\ell},u(p^{\ell}))|\leq A\sum\limits_{\ell=m}^{\infty}p^{-(\beta+\gamma)\ell}<\infty,$$
and if $\alpha=1$:
$$\sum\limits_{\ell=m}^{\infty}\ell|\widetilde{f}(p^{\ell},u(p^{\ell}))|\leq A\sum\limits_{\ell=m}^{\infty}\ell p^{-(\beta+\gamma)\ell}<\infty.$$
\end{proof}

\begin{theorem}\label{theorem3}
Under the conditions \eqref{0-19}, \eqref{0-21}, the solution to the equation \eqref{0-8}, obtained by iteration process with subsequent continuation, satisfies the equation \eqref{0-1}.
\end{theorem}

\begin{proof}
It follows from Proposition \ref{pr5}, Lemma \ref{lemma1}, and Lemma \ref{lemma3} that $D^{\alpha}u$ exists, $D^{\alpha}I^{\alpha}\widetilde{f}$ exists and $D^{\alpha}I^{\alpha}\widetilde{f}=\widetilde{f}$. Hence we can apply the operator $D^{\alpha}$ to the equation \eqref{0-8} and obtain

$$(D^{\alpha}u)(|t|_p)=D^{\alpha}u_0+D^{\alpha}I^{\alpha}\widetilde{f}(|\cdot|_p,u(|\cdot|_p))(|t|_p),$$
which is equivalent to
$$(D^{\alpha}u)(|t|_p)=\widetilde{f}(|\cdot|_p,u(|\cdot|_p))(|t|_p),$$
since $D^{\alpha}$ is equal to zero on constant functions.
\end{proof}

\end{document}